\documentclass[a4paper]{amsart}

\usepackage{amsmath,amssymb,amsbsy,amsfonts,amsthm,latexsym, old-arrows,
	amsopn,amstext,amsxtra,euscript,amscd,mathrsfs}
\usepackage{amsmath,amssymb,amsbsy,amsfonts,latexsym,amsopn,amstext,cite,
	amsxtra,euscript,amscd,bm,bbm}
\usepackage{parskip}
\usepackage{mathtools}
\usepackage{todonotes}
\usepackage{url}
\usepackage[colorlinks,linkcolor=blue,anchorcolor=blue,citecolor=blue,backref=page]{hyperref}

\renewcommand*{\backref}[1]{}
\renewcommand*{\backrefalt}[4]{%
	\ifcase #1 (Not cited.)%
	\or        (p.\,#2)%
	\else      (pp.\,#2)%
	\fi}
\begin{document}

\newtheorem{counter}{Theorem}
\newcommand{\jed}{\mbox{\boldmath$1$}}
\newcommand{\To}{\longrightarrow}
\newcommand{\A}{\mathcal{A}}
\newcommand{\C}{\mathbb{C}}
\newcommand{\Cc}{\mathcal{C}}
\renewcommand{\S}{\mathcal{S}}
\newcommand{\F}{\mathcal{F}}
\newcommand{\R}{\mathbb{R}}
\newcommand{\E}{\mathbb{E}}
\newcommand{\1}{\mathbb{1}}
\renewcommand{\P}{\mathbb{P}}
\newcommand{\Z}{\mathbb{Z}}
\newcommand{\N}{\mathbb{N}}
\newcommand{\W}{\mathbb{W}}
\newcommand{\T}{\mathbb{T}}
\newcommand{\M}{\mathcal{M}}
\newcommand{\U}{\mathcal{U}}
\newcommand{\V}{\mathcal{V}}
\newcommand{\Q}{\mathcal{Q}}
\renewcommand{\S}{\mathcal{S}}
\newcommand{\B}{\mathcal{B}}
\newcommand{\ce}{\mathbb E}
\newcommand{\I}{\mathcal{I}}
\newcommand{\setdef}{\stackrel {\rm {def}}{=}}
\newcommand{\rea}{\text{Re}}
\newcommand{\ds}{\displaystyle}
\newcommand{\script}{\scriptstyle}
\newcommand{\tif}{\tilde f}
\newcommand{\tig}{\tilde g}
\newtheorem*{thank}{\ \ \ \textbf{Acknowledgment}}
\newcounter{tictac}
\newenvironment{fleuveA}{
	\begin{list}{$(\emph{\textbf{A\arabic{tictac}}})$}{\usecounter{tictac}
			\leftmargin 1cm\labelwidth 2em}}{\end{list}}
\def\1{\,\rlap{\mbox{\small\rm 1}}\kern.15em 1}
\def\ind#1{\1_{#1}}
\def\build#1_#2^#3{\mathrel{\mathop{\kern 0pt#1}\limits_{#2}^{#3}}}
\def\tend#1#2{\build\hbox to 12mm{\rightarrowfill}_{#1\rightarrow #2}^{ }}
\def\cor#1{\build\longlefrightarrow_{}^{#1}}
\def\tendn{\tend{n}{\infty}}
\def\converge#1#2#3#4{\build\hbox to
	#1mm{\rightarrowfill}_{#2\rightarrow #3}^{\hbox{\scriptsize #4}}}


\newcommand{\printdate}{\today}
\theoremstyle{definition}

\newtheorem{thm}{Theorem}[section]
\newtheorem{prop}[thm]{Proposition}
\newtheorem{lem}[thm]{Lemma}
\newtheorem{defn}[thm]{Definition}
\newtheorem{notation}[thm]{Notation}
\newtheorem{example}[thm]{Example}
\newtheorem{conj}[thm]{Conjecture}
\newtheorem{prob}[thm]{Problem}
\newtheorem{Prop}[thm]{Proposition}
\newtheorem{que}[thm]{Question}
\newtheorem{rem}[thm]{Remark}
\newtheorem{rems}[thm]{Remarks}

\newtheorem{Cor}[thm]{Corollary}
\newtheorem{thmnonumber}{Theorem}
\newtheorem{fact}[thmnonumber]{Fact}
\newcommand{\cb}{{\mathcal B}}
\newcommand{\ca}{{\mathcal A}}
\newcommand{\cc}{{\mathcal C}}
\newcommand{\cd}{{\mathcal D}}
\newcommand{\cf}{{\mathcal F}}
\newcommand{\ch}{{\mathcal H}}
\newcommand{\cm}{{\mathcal M}}
\newcommand{\hatca}{\widehat{\mathcal A}}
\newcommand{\hatcb}{\widehat{\mathcal B}}
\newcommand{\hatcc}{\widehat{\mathcal C}}
\newcommand{\hatcd}{\widehat{\mathcal D}}
\newcommand{\ot}{\otimes}
\newcommand{\la}{\lambda}
\newcommand{\tf}{T_{\varphi}}
\newcommand{\va}{\varphi}
\newcommand{\vep}{\varepsilon}
\newcommand{\ov}{\overline}
\newcommand{\un}{\underline}
\newcommand{\ux}{\underline{X}}
\newcommand{\uy}{\underline{Y}}
\newcommand{\hatuy}{\widehat{\underline{Y}}}
\newcommand{\hatux}{\widehat{\underline{X}}}
\newcommand{\uz}{\underline{Z}}
\newcommand{\ka}{\mathcal{K}}
\newcommand{\si}{\sigma}

\newcommand{\bfu}{\boldsymbol{f}}
\newcommand{\bh}{\boldsymbol{h}}
\newcommand{\bb}{\boldsymbol{b}}
\newcommand{\bc}{\boldsymbol{c}}
\newcommand{\bx}{\boldsymbol{x}}
\newcommand{\aaa}{\boldsymbol{a}}
\newcommand{\bmu}{\bm \mu}
\newcommand{\bmnu}{\bm \nu}
\newcommand{\bml}{\bm \lambda}
\newcommand{\bun}{\boldsymbol{1}}

\newtheorem*{xrem}{Remark}
\newtheorem*{xconj}{Conjecture}
\newtheorem*{xques}{Question}
\newcommand{\intBohr}{\mathop{\mathlarger{\mathlarger{\mathlarger{\landupint}}}}}
\newcommand{\Prod}{\mathop{\mathlarger{\mathlarger{\mathlarger{\prod}}}}}
\newcommand{\Sum}{\mathlarger{\mathlarger{\sum}}}
\def\Max#1{\mathlarger{\max_{#1}}}
\def\Min#1{\mathlarger{\min_{#1}}}
\def\Sup#1{\mathlarger{\sup_{#1}}}
\def\En#1{(\lfloor #1 \rfloor)}

\def \d{{\rm d}}
\def \f{\overline{f}}
\def \g{\overline{g}}
\def \e{\varepsilon}
\def\InvFourier#1{\stackrel{\vee}{#1}}
\def\Invstar#1{\widetilde{#1}}
\def \Cov{\text{Cov}}
\def\Pr{\mathbb P}
\newcommand{\beq}{\begin{equation}}
\newcommand{\eeq}{\end{equation}}
\newcommand{\xbm}{(X,{\mathcal B},\mu)}
\newcommand{\xbmt}{(X,{\mathcal B},\mu,T)}
\newcommand{\tilxbm}{(\tilde{X},\tilde{{\cal B}},\tilde{\mu})}
\newcommand{\zdr}{(Z,{\mathcal D},\rho)}
\newcommand{\ycn}{(Y,{\mathcal C},\nu)}
\title[On the homogeneous ergodic bilinear $\cdots$"]{ On the homogeneous ergodic bilinear averages with $1$-bounded multiplicative weights$^{*}$}
\author{E. H. \lowercase{el} Abdalaoui\\
	$^{*}$  D\lowercase{edicated to} P\lowercase{rofessor}  M\lowercase{ichael} L\lowercase{in  on the occasion of his 86th birthday.}}








\medskip
{\renewcommand\abstractname{Abstract}
\begin{abstract} We establish a generalization of Bourgain double recurrence theorem  and ergodic Bourgain-Sarnak's theorem by proving that for any aperiodic $1$-bounded multiplicative function  $\boldsymbol{\nu}$, for any map $T$  acting on a probability space $(X,\A,\mu)$, for any integers $a,b$ ,  for any $f,g \in L^2(X)$, and for almost all $x \in X$, we have	
	\[
	\frac{1}{N}
	\sum_{n=1}^{N}\boldsymbol{\nu}(n)
	f(T^{a n}x)g(T^{bn}x) \tend{N}{+\infty} 0.
	\]
We further present with proof the key ingredients of Bourgain's proof of his double recurrence theorem.
\end{abstract}}



\maketitle
\section{Introduction}

The purpose of this paper is to consider the homogeneous ergodic bilinear averages with arithmetical weight. But, we focus only on the case when the weight is given by $1$-bounded multiplicative function. This will give an another generalization of Bourgain's double recurrence theorem (BDRT) \cite{Bourgain-D} and  Bourgain-Sarnak's ergodic theorem \cite{Sarnak}. For a simple proof of this later theorem, we refer to \cite{elabdalDCDS}. But, we stress that our proof follows that of Bourgain. As a consequence, we present with proof its key ingredients. 

A careful extension and understanding of Bourgain's proof gives that for any invertible measure preserving transformation $T$, acting on a probability space $(X, \B, \mu)$, for any $f \in L^{r}(X,\mu)$ , $g \in L^{r'}(X,\mu)$ such that $\frac{1}{r}+\frac{1}{r'}= 1$, for any non-constant polynomials $P(n), n \in \Z$, taking integer values, and for almost all $x \in X$, we have,
\begin{align*}
\frac{1}{N}\sum_{n=1}^{N}f(T^{P(n)}x) g(T^{n}x).
\end{align*}
 (See \cite{TMT}, see also Lemma \ref{WW-Gowers} (equation \eqref{WW-Nil}) and Lemma \ref{Bourgain-F}).\\
 
The study of the almost everywhere convergence of the homogeneous ergodic bilinear averages with weight  was initiated by  I. Assani, D. Duncan, and R. Moore in \cite{AssaniDM}. Therein, the authors proved a Wiener-Wintner version of BDRT, that is, the exponential sequences $(e^{2\pi i nt})_{n \in \Z}$ are good weight for the homogeneous ergodic bilinear averages. 
Subsequently, I. Assani and R. Moore  showed that the polynomials exponential sequences $\big(e^{2\pi i P(n)}\big)_{n \in \Z}$  are also uniformly good weights for the homogeneous ergodic bilinear averages
\cite{AssaniM}. One year later, I. Assani \cite{Assani-Nil} and P. Zorin-Kranich \cite{Zorich} proved independently that the nilsequences are
uniformly good weights for the homogeneous ergodic bilinear averages. Their proofs depend also on Bourgain's theorem. Very recently, the author extended Bourgain-Sarnak theorem by proving that the M\"{o}bius and  Liouville functions are a good weight for the homogeneous ergodic bilinear averages \cite{IJL}. But there is a gap in the proof. Here, we will generalize that theorem to the bounded multiplicative function and we will thus fill the gap. Let us point out that therein, the author extended Assani's theorem \cite{Assani} by proving , independently of the gap, that the M\"{o}bius and Liouville functions are a good weight for the homogeneous F\"{u}rstenberg's  multilinear ergodic averages provided that the map is  weakly mixing and its restriction to its Pinsker algebra has singular spectrum.  

Here, our proof follows closely  Bourgain's proof of his double recurrence theorem which in turn is based on his method and the Calder\'{o}n transference principal. For a nice account on this method, we refer to \cite{Nair}, \cite{Thouvenot}. Despite the fact that the classical spectral analysis can not be applied to study  F\"{u}rstenberg's multilinear ergodic average, some kind of spectral analysis tools based on the Fourier transform can be adapted for its studies. This is accomplished by applying Cald\'eron principal and the discrete Fourier transform which can be seen as a spectral isomorphism. We notice that in this setting, the dynamics on the diagonal in F\"{u}rstenberg's ergodic average is interpreted as an operation on the kernels. The kernel is an average mass on the particles $X=\big\{x_1,\cdots,x_N \big\}$ and the operation act on the diagonal of $X \times X$  and it is given as a kernel the Dirac mass  $(x_i,x_i)$, $i=1,\cdots,N$. We emphasize that the product on the observable functions in $\ell^2$ is interpreted as a convolution. We thus believe that, as in Bourgain's proof, the harmonic analyis methods can be useful to address the problem of the convergence almost everywhere of F\"{u}rstenberg's multilinear ergodic average. This approach is raised in \cite{elabdal-poly}. 

 We recall that the problem of the convergence almost everywhere (a.e.) of the ergodic  multilinear averages was initiated by F\"{u}rstenberg in \cite[Question 1 p.96]{Fbook}. Bourgain answered that question by proving the following:

Let $T$ be a map acting on a probability space $(X,\mathcal{A},\nu)$, and $a,b \in\mathbb{Z}$,
then for any $f,g \in L^\infty(X)$, the averages

\[
\frac{1}{N}
\sum_{n=1}^{N}
f(T^{an}x)g(T^{bn}x)
\]
converge for almost every $x$.
Here, we will state and formulate the fundamental ingredients in Bourgain's proof of his theorem. We will also present the proof of the main ingredients. In fact, our proof is essentially based on those ingredients.
 
\section{Set up and Tools}
An arithmetical function is a complex-valued function defined on the positive
integers. It is also called a number-theoretic function. The function $\bmnu$  is called multiplicative if $\bmnu$ is not identically zero and if
\[ \bmnu(nm)=\bmnu(n)\bmnu(m) \textrm{~~whenever~~} (n,m)=1.\]
$(n,m)$ stand for the greatest common divisor of $n$ and $m$.

\noindent $\bmnu$ is called completely multiplicative if we also have
$$\bmnu(nm)=\bmnu(n)\bmnu(m),~~~~~~~\forall n,m \in \N.$$

An easy example of function which is multiplicative but not completely multiplicative can be given by the function $f(n)=\lfloor \sqrt{n}\rfloor-\lfloor\sqrt{n-1}\rfloor$, where, 
as customary, $\lfloor~.~\rfloor$ is the integer part. 

For the bounded completely multiplicative,  the Liouville function is the famous example. This function  is defined for the positive integers $n$ by
$$
\bml(n)=(-1)^{\Omega(n)},
$$
where $\Omega(n)$ is the length of the word $n$ is the alphabet of prime, that is, $\Omega(n)$ is the number of prime factors of $n$ counted with multiplicities.\\
Its allies is the M\"{o}bius function which is only multiplicative. It is given by
\begin{equation}\label{Mobius}
\bmu(n)= \begin{cases}
1 {\rm {~if~}} n=1; \\
\bml(n)  {\rm {~if~}} n
{\rm {~is~the~product~of~}} r {\rm {~distinct~primes}}; \\
0  {\rm {~if~not}}
\end{cases}
\end{equation}
These two functions are of great importance in number theory since the Prime Number Theorem is equivalent to \begin{equation}\label{E:la}
\sum_{n\leq N}\bml(n)={\rm o}(N)=\sum_{n\leq N}\bmu(n).
\end{equation}
Furthermore, there is a connection between these two functions and Riemann zeta function, namely
$$
\frac1{\zeta(s)}=\sum_{n=1}^{\infty}\frac{\bmu(n)}{n^s} \text{ for any }s\in\mathbb{C}\text{ with }\rea(s)>1.
$$
Moreover, Littlewood proved that the estimate
\[
\left|\ds \sum_{n=1}^{x}\bmu(n)\right|=O\left(x^{\frac12+\varepsilon}\right)\qquad
{\rm as} \quad  x \longrightarrow +\infty,\quad \forall \varepsilon >0
\]
is equivalent to the Riemann Hypothesis (RH) (\cite[pp.315]{Titchmarsh}).

We recall that the proof of Sarnak-Bourgain theorem \cite{Sarnak} is based on the following Davenport-Hua's estimation \cite{Da},  \cite[Theorem 10.]{Hua}: for each $A>0$, for any $k \geq 1$, we have
\begin{equation}\label{vin}
\max_{z \in \T}\left|\displaystyle\sum_{n \leq N}z^{n^k}\bmu(n)\right|\leq C_A\frac{N}{\log^{A}N}\text{ for some }C_A>0.
\end{equation}
We refer to \cite{elabdalDCDS} and \cite{Cuny-Weber} for this proof. 



\noindent The inequality \eqref{vin}  can be established also for the Liouville function by applying carefully the following identity:

\[\bml(n)=\sum_{d:d^2 | n}\bmu\Big(\frac{n}{d^2}\Big).\]
We further recall that the multiplicative function $\bmnu$ is said to be aperiodic if

$$\frac{1}{N}\sum_{n=1}^{N}\bmnu(an+b) \tend{N}{+\infty}0,$$
\noindent for any $(a, b) \in \N^* \times \N$.
By Davenport's theorem \cite{Da} and Bateman-Chowla's theorem \cite{Batman-Cho}, the M\"{o}bius and Liouville functions are aperiodic.

\section{Wiener-Wintner's version of Daboussi-Katai-Bourgain-Sarnak-Ziegler's (DKBSZ) criterion (WWDKBSZ)}
For the proof of our main result, we need a straightforward generalization \`a la Wiener-Wintner of the so-called DKBSZ criterion. This criterion is based on the results of Bourgain-Sarnak-Ziegler \cite[Theorem 2]{BSZ}, and Katai \cite{Katai} , which in turn  develop some ideas of Daboussi (presented in \cite{Daboussi}, \cite{DaboussiII}).  we state now the WWDKBSZ criterion in the following form.

\begin{Prop}[WWKBSZ criterion]\label{WWKBSZ}
	Let $(X,\ca,\mu)$ be a Lebesgue probability space and $T$ be an invertible measure preserving transformation. 
	Let $\bmnu$ be a multiplicative function, $f$ be in $L^{\infty}$ with $\|f\|_{\infty} \leq 1$ and
	$\varepsilon>0$. Suppose that for almost all point $x\in X$ and for all different prime numbers $p$ and $q$ less than $\exp(1/\varepsilon)$, we have
	\begin{equation}
	\label{eq:f_k limit}
	\limsup_{N\to\infty}\sup_{t}\left| \dfrac{1}{N}\sum_{n=1}^N e^{2 \pi i n (p-q) t}f (T^{pn}x) f(T^{qn}x) \right| < \varepsilon,
	\end{equation}
	then, for almost all $x \in X$, we have
	\begin{equation}
	\label{eq:almost orthogonality for f_k}
	\limsup_{N\to\infty} \sup_{t}\left|\dfrac{1}{N} \sum_{n=1}^N \bmnu(n) e^{2 \pi i n t}f(T^{n}x) \right| < 2\sqrt{\varepsilon\log1/\varepsilon}.
	\end{equation}
\end{Prop}
\begin{proof}The proof is verbatim the same as  \cite[Theorem 2]{BSZ},
	except that at the equation (2.7) one need to apply the following elementary inequality: for any two bounded positive functions $F$ and $G$, we have
	$$\sup(F(x)+G(x)) \leq \sup(F(x))+\sup(G(x)).$$
\end{proof}
\subsection{Gowers Norms} 
Gowers norms are a great tools  in additive number theory, combinatorics and ergodic theory. We present its definition for the shift on the set of integers.

Let $d \geq 1$ and $C_d=\{0,1\}^d$. If $\bh \in \Z^d$ and $\bc \in C_d$, then 
$\bc.\bh=\sum_{i=1}^{d}c_ih_i.$  Let ${\big(f_{\bc}\big)}_{\bc \in C_d}$ be a family of bounded functions that are finitely supported, 
that is, for each $\bc \in C_d$, $f_{\bc}$ is in $L_c^{\infty}(\Z)$ the subspace of functions that are finitely supported.  The Gowers inner product is given by 
$$\Big\langle \big(f_{\bc}\big)\Big \rangle_{U^d(\Z)}=\sum_{\bx,\bh \in G^{d+1}}\prod_{\bc \in C_d}\Cc^{|\bc|}f_{\bc}(g+\bc.\bh)
,$$
where $|\bc|=\bc.\boldsymbol{1}$, $\boldsymbol{1}=(1,1,\cdots,1) \in C_d$ and $\Cc$ is the conjugate anti-linear operator.
If all $f_{\bc}$ are the same function $f$ then the Gowers norms of $f$ is defined by
$$\big|\big|f\big|\big|_{U^d(\Z)}^{2^d}=\Big\langle \big(f\big)\Big \rangle_{U^d(\Z)}.$$
The fact that $\big|\big|.\big|\big|_{U^d(\Z)}$ is a norm for $d \geq 2$ follows from the following generalization of Cauchy-Bunyakovski-Schwarz inequality for the Gowers inner product.
\begin{prop}(Cauchy-Bunyakovskii-Gowers-Schwarz inequality)\label{Tao}
	$$ \Big\langle \big(f_c\big)\Big \rangle_{U^d(\Z)} \leq \prod_{c \in C_q}\big|\big|f_c\big|\big|_{U^d(\Z)}.$$
\end{prop}
The proof of Cauchy-Bunyakovskii-Gowers-Schwarz inequality can be obtain easily by applying inductively Cauchy-Bunyakovskii-Schwarz inequality. Indeed, it easy to check that we have
$$ \Big\langle \big(f_c\big)\Big \rangle_{U^d(\Z)} \leq \prod_{j=0,1}
\Big|\langle \big(f_{\pi_{i,j}}(\bc)\big)\Big \rangle_{U^d(G)}\Big|^{\frac{1}{2}},$$
For all $i=1,\cdots, d-1$, where $\pi_{i,j}(\bc) \in C_d$ is formed from $\bc$ by replacing  the $i^{\rm{th}}$ coordinates with $j$. 
Iterated this, we obtain the complete proof of Proposition \ref{Tao}.

Combining Cauchy-Bunyakovskii-Gowers-Schwarz inequality with the binomial formula and the multilinearity of the Gowers inner
product one can easily check that the triangle inequality for $\big\|.\big\|_{U^d(G)}$ holds.  We further have
\begin{align*}
\big\|f\big\|_{U^d(\Z)} \leq \big\|f\big\|_{U^{d+1}(\Z)},
\end{align*}
by applying Cauchy-Bunyakovskii-Gowers-Schwarz inequality with $f_{\bc}=1$ if $\bc_d=0$ and $f_{\bc}=1$ if $\bc_d=1$. The Gowers norms are also invariant under the shift and conjugacy.

We define also the discrete derivative of function $f :\Z \rightarrow \C$ by putting
$$\partial_h(f)=f(x+h).\overline{f}(x),$$
for all $h,x \in \Z$. We can thus write the Gowers norm of $f$ as follows

\begin{eqnarray*}
\big\|f\big\|_{U^d(\Z)}^{2^d}=\int_{G^{d+1}}\partial_{h_1}\partial_{h_2}\cdots\partial_{h_d}(f)(x) d\bh dx.
\end{eqnarray*}
If further $f$ take values on $\R/\Z$ and 
$$\partial_{h_1}\partial_{h_2}\cdots\partial_{h_{d+1}}(f)(x)=0,$$ 
for all $h_1,\cdots,h_{d+1},x \in G$, then $f$ is said to be a polynomial function of degree at most $d$. The degree of $f$ is denoted by
$d^{\circ}(f)$.

According to this it is easy to see that for any function $f$ and any polynomial function $\phi$ of degree at most $d$, we have
$$\big|\big|e^{2 \pi i \phi(x)}f(x)\big|\big|_{U^d(\Z)}=\big|\big|f\big|\big|_{U^d(\Z)}.$$
Therefore
\begin{eqnarray}\label{vander}
\sup_{\phi,~~d^{\circ}(\phi)\leq d}\Big|\int_G e^{2 \pi i \phi(x)}f(x)dx\Big| \leq \big|\big|f\big|\big|_{U^d(\Z)}.
\end{eqnarray}

In application and here we need to define the Gowers norms for a bounded function defined on $\big\{1,\cdots,N\big\}$. For that, 
if $f$ is a bounded function defined on $\big\{1,\cdots,N-1\big\}$, we put
\[\big|\big|f\big|\big|_{U^d[N]}=
\frac{\big|\big|\widetilde{f}\big|\big|_{U^d(\Z/2^d.N\Z)}}{\big|\big|\1_{[N]}\big|\big|_{U^d(\Z/2^d.N\Z)}},\] 
where $\widetilde{f}=f(x).\1_{[N]},$ $x \in \Z/2^d.N\Z$, $\1_{[N]}$ is the indicator function of $\big\{1,\cdots,N\big\}$.
For more details on Gowers norms, we refer to \cite{TaoHfourier},\cite{TaoAdditive}.

The sequence $f$ is said to have a small Gowers norms if for any $d \geq 1$, 
\[\big|\big|f\big|\big|_{U^d[N]} \tend{N}{+\infty}0. \]
An example of sequences of small Gowers norms that we shall need here is given by Thue-Morse and Rudin-Shapiro sequences.  This result is due J. Konieczny \cite{Konieczny}.

The notion of Gowers norms can be extended to the dynamical systems as follows.

Let $T$ is be an ergodic measure preserving transformation on $X$. Then,  for any $k \geq 1$,
the Gowers seminorms on $L^{\infty}(X)$ are defined inductively as follows
$$\||f|\|_1=\Big|\int f d\mu\Big|;$$
$$\||f|\|_{k+1}^{2^{k+1}}=\lim\frac{1}{H}\sum_{l=1}^{H}\||\overline{f}.f\circ T^l|\|_{k}^{2^{k}}.$$

For each $k\geq 1$, the seminorm $\||.|\|_{k}$ is well defined. Those norms are also called Gowers-Host-Kra's seminorms. For details, we refer the reader to \cite{Host-K2}. Notice that
the definitions of Gowers seminorms can be also easily extended to non-ergodic maps.

The importance of the Gowers-Host-Kra's seminorms in the study of the nonconventional multiple ergodic averages is due to the existence of $T$-invariant sub-$\sigma$-algebra $\mathcal{Z}_{k-1}$ of $X$ that satisfies
$$\E(f|\mathcal{Z}_{k-1})=0 \Longleftrightarrow \||f|\|_{k}=0.$$
This was proved by Host and Kra in \cite{Host-K2}. The existence of the factors $\mathcal{Z}_{k}$ was established by Host and Kra and independently by Ziegler in \cite{Ziegler}.

At this point, we are able to state the second main ingredient need it in our proof. It is due to Assani-Dacuna and Moore \cite{AssaniDM}. This result extend à la Wiener-Wintner Bourgain's double recurrence theorem (WWBDRT).  
\begin{Prop}\label{Assani-Moore} Let $(X,\mu,T)$ be an ergodic	
	dynamical system and $a,b$ distinct non-zero integers. Then,
	for any $f,g \in L^{\infty}$ with $\min\{\|f_1\|_{U^{2}},\|g\|_{U^{2}}\}=0$,  there exist a measurable set $X'$ of full measure such
	,for any $x \in X'$,  we have
	$$ {\sup_{|z|=1}}\Big|{\frac1{N}}\sum_{n=0}^{N-1}z^nf(T^{an}x)g(T^{bn}x) \Big|
	\tend{N}{+\infty}0.$$
\end{Prop}
\noindent{}Proposition \ref{Assani-Moore} has been extended to the nilsequences independently by I. Assani \cite{Assani-Nil} and A. Zorin-Kranich\cite{Zorich}. We refer to \cite{Assani-Nil} or \cite{Zorich} for the definition of nilsequences.

\noindent{}We need also the following straightforward lemma.
\begin{lem}\label{WW-Gowers}Let $T$ is be an ergodic measure preserving transformation on $(X,\B,\mu)$ and
	$f \in L^\infty(X,\mu)$. Then
\begin{eqnarray}
	\int \limsup_{N \longrightarrow +\infty}\Big|{\frac1{N}}\sum_{n=0}^{N-1}z^nf(T^{n}x) \Big| d\mu(x)
\leq \|f\|_{U^{2}}.
\end{eqnarray}
\end{lem}
The previous lemma can be extended to the nilsequences. We thus have, for any $k$-nilsequence $(b_n)$.
\begin{eqnarray}\label{WW-Nil}
\int \limsup_{N \longrightarrow +\infty}\Big|{\frac1{N}}\sum_{n=0}^{N-1}b_nf(T^{n}x) \Big| d\mu(x)
\leq \|f\|_{U^{k+1}}.
\end{eqnarray}
The proof can be obtained  by induction and by applying the classical van der Corput inequality \cite[p.25]{KN}.   
\section{Some tools on the oscillation method and Calder\'{o}n transference principle}
 Let $k \geq 2$ and $(X,\ca,T_i,\mu)_{i=1}^{k}$ be a family of dynamical systems, that is, for each $i=1,\cdots, k$, $T_i$ is a measure-preserving transformation, ($\forall A \in \ca$, $\mu(T_i^{-1}A)=\mu(A)$.). The sequence of complex number $(a_n)$ is said to be good weight in $L^{p_i}(X,\mu)$, $p_i\geq1$, $i=1,\cdots,k$, with 
 $\sum_{i=1}^{k}\frac1{p_i}=1,$ if for any $f_i \in L^{p_i}(X,\mu)$, $i=1,\cdots,k$,
 the ergodic $k$-multilinear averages   
 $$\frac1{N}\sum_{j=1}^{N}a_j\prod_{i=1}^{k}f_i(T_i^jx)$$
 converges a.e.. The maximal multilinear ergodic inequality is said to hold in $L^{p_i}(X,\mu)$, $p_i\geq1$, $i=1,\cdots,k$, with 
 $\sum_{i=1}^{k}\frac1{p_i}=1,$ if for any $f_i \in L^{p_i}(X,\mu)$, $i=1,\cdots,k$, the maximal function given by 
 $$M(f_1,\cdots,f_k)(x)=\sup_{N \geq 1}\Big|\frac{1}{N}\sum_{j=1}^{N}a_j\prod_{i=1}^{k}f_i(T_i^jx)\Big|$$
 satisfy the weak-type inequality 
 $$\sup_{\lambda > 0}\Bigg(\lambda \mu\Big\{x~~:~~M(f)(x)>\lambda \Big\}\Bigg) \leq C \prod_{i=1}^{k}\big\|f_i\big\|_{p_i},$$
 where $C$ is an absolutely constant.
 
 As far as the author know, it is seems that it is not known whether the classical maximal multilinear ergodic inequality ($a_n=1$, for each $n$) holds 
 for the general case $n \geq 3$. Nevertheless, we have the following Calder\'{o}n transference principal in the homogeneous case. 
 
\begin{Prop}\label{CalderonP} Let $(a_n)$ be a sequence of complex number and assume that for any $\phi,\psi \in \ell^2(\Z)$, for any non-constant polynomials $P,Q$ mapping natural numbers to themselves, for any $1 \leq p,q,r \leq +\infty$  such that $\frac{1}{r}=\frac{1}{p}+\frac{1}{q}$,  we have
	\begin{align*}
			\Big\|\sup_{ N  \geq 1}\Big|\frac1{N}\sum_{n=1}^{N}a_n \phi(j+P(n))\psi(j+Q(n)) 
		\Big|\Big\|_{\ell^r(\Z)} \\\leq  C.
		\big\|\phi\big\|_{\ell^p(\Z)}\big\|\psi\big\|_{\ell^q(\Z)},
	\end{align*}

	where $C$ is an absolutely constant. Then, for any dynamical system $(X,\A,T,\mu)$, for any $f \in L^p(X,\mu)$ and $g \in L^q(X,\mu)$ , we have 
		\begin{align*}
	\Big\|\sup_{N \geq 1}\Big|\frac1{N}\sum_{n=1}^{N}a_n f(T^{P(n)}x)g(T^{Q(n)}x) \Big|\Big\|_r\\ \leq
	C \big\|f\big\|_{p}\big\|g\big\|_{q}.
	\end{align*}
\end{Prop}
\noindent We further have
\begin{Prop}\label{CalderonP2} Let $(a_n)$ be a sequence of complex number and  assume that for any $\phi,\psi \in \ell^2(\Z)$, for any $\lambda>0$, for any integer $J \geq 2$, for any non-constant polynomials $P,Q$ mapping natural numbers to themselves, for any $1 \leq p,q\leq +\infty$  such that $\frac{1}{p}+\frac{1}{q}=1$, , we have
\begin{align*}
	\sup_{\lambda > 0} \Big(\lambda\Big|\Big\{1 \leq j \leq J~~:~~ \sup_{ N  \geq 1}\Big|\frac1{N}\sum_{n=1}^{N}a_n \phi(j+P(n))\psi(j+Q(n)) 
	\Big|> \lambda \Big\}\Big|\Big) \\\leq C \big\|\phi\big\|_{\ell^p(\Z)}\big\|\psi\big\|_{\ell^q(\Z)},
\end{align*}
	where $C$ is an absolutely constant. Then, for any dynamical system $(X,\A,T,\mu)$, for any $f,g \in L^2(X,\mu)$, we have 
\begin{align*}	
	\sup_{\lambda> 0} \Big(\lambda \mu\Big\{x \in X~~:~~\sup_{N \geq 1}\Big|\frac1{N}\sum_{n=1}^{N}a_n f(T^{P(n)}x)g(T^{Q(n)}x) \Big| > \lambda \Big\}\Big)\\ \leq  C
	\big\|f\big\|_{2}.\big\|g\big\|_{2}.
\end{align*} 
\end{Prop}

\noindent It is easy to check that Proposition \ref{CalderonP} and \ref{CalderonP2} hold for the homogeneous $k$-multilinear ergodic averages, for any 
$k \geq 3$. Moreover, it is easy to state and to prove the finitary version where $\Z$ is replaced by $\Z/J\Z$ and the functions $\phi$ and 
$\psi$ with $J$-periodic functions. 
\section{Main result and its proof}
The subject of this section is to state and to prove the main result of this paper and its consequences. We begin by stating our main result. 

\begin{thm}\label{Mainofmain}Let $(X,\ca,\mu,T)$ be an ergodic dynamical system, let $\bmnu$ be an aperiodic $1$-bounded multiplicative function and $a,b \in \Z$. Then, for any $f,g \in L^2(X)$, for almost all $x \in X$,
	$$\frac1{N}\sum_{n=1}^{N}\bmnu(n) f(T^{an}x)g(T^{bn}x) \tend{N}{+\infty}0,$$
\end{thm}

\noindent Consequently, we obtain the following corollary. 

\begin{Cor}[\cite{IJL}]Let $(X,\ca,\mu,T)$ be an ergodic dynamical system, and $T_1,T_2$ be powers of $T$. Then, for any $f,g \in L^2(X)$, for almost all $x \in X$,
	$$\frac1{N}\sum_{n=1}^{N}\bmnu(n) f(T_1^nx)g(T_2^nx) \tend{N}{+\infty}0,$$
	where $\bmnu$ is the Liouville function or the M\"{o}bius function.
\end{Cor}

Before proceeding to the proof of Theorem \eqref{Mainofmain}, we recall the following notations.

Let $T$ be a map acting on a probability space $(X,\mathcal{A},\mu)$ and $\bmnu$ be an aperiodic $1$-bounded multiplicative function. For any any $\rho>1$, we will denote by $I_\rho$ the set $\Big\{\En{\rho^n}, n \in \N \Big\}.$ The maximal functions are defined by 
\begin{align*}
M_{N_0,\bar{N}}(f,g)(x)=
\sup_{\overset{ N_0 \leq N \leq \bar{N}}{N \in I_\rho}}\Big|\frac1{N}\sum_{n=1}^{N}\bmnu(n) f(T^{n}x) g(T^{-n}x)\Big|,
\end{align*}
and
\begin{align*}
M_{N_0}(f,g)(x)=
\sup_{\overset{N \geq N_0}{N \in I_\rho}}\Big|\frac1{N}\sum_{n=1}^{N}\bmnu(n) f(T^{n}x) g(T^{-n}x)
\Big|.
\end{align*}

Obviously, 
\[\lim_{\bar{N} \longrightarrow +\infty}M_{N_0,\bar{N}}(f,g)(x)=M_{N_0}(f,g)(x).\]

For the shift $\Z$-action, the maximal functions are denoted by $m_{N_0,\bar{N}}(\phi,\psi)$ and $m_{N_0}(\phi,\psi)$.
We denote by $\ell^p(J),$  $p>1$, the subspace of  $\ell^p(\Z),$ of the observable functions with finite support subset of $[0,J].$ For each $0<\delta<1$, we define a function $\sigma_\delta$ by
\[\sigma_\delta(t)=\theta(\delta,t)-\theta(2\delta,t),\]
where $\theta(\delta,t)$ is the half of bumpfunction with support $(\delta, 2\delta]$ and
$\theta(\delta,2\delta)=1.$ 
We thus get 
$$\sum_{\overset{0<\delta<1} {\textrm{~~dyadic~~}}}\sigma_\delta(t)=1, ~~~~~\textrm{for~~each~~} t \in (0,1].$$
For any $N \in \N^*$, $x \in [1,J]$ and $\theta \in [0,1)$, we put
$$P_{x,N}(\theta)=\frac{1}{N}\sum_{n=x-N}^{x-1}e^{2 \pi i n \theta} \psi(n),$$
we further set
$$\bmnu_x(n)=\bmnu(n+x), ~~~~~~~\forall n \in \N.$$
\noindent{}For a finite subset $E$ of the torus $\T$ and $\epsilon>0$, we denote its $\epsilon$-neighbourhood by $E(\epsilon)$. We recall that $E(\epsilon)$ is given by
$$E(\epsilon)=\Big\{\lambda \in \T~~:~~~\min_{\gamma \in E}\big|\lambda-\gamma\big| < \epsilon\Big\}.$$

\noindent{}We denote by $\mathbb{A}(\T)$  the Weiner algebra of the absolutely convergent Fourier series equipped with the norm $\||.\||$. We recall that for any $\varphi \in \mathbb{A}(\T)$, we have  
$$ \big\|\big|\varphi\big\|\big|_{\mathbb{A}}=\sum_{ n \in \Z}\big|\widehat{\varphi}(n)\big|.$$
\noindent The Fourier transform on $\ell^2(\Z)$ will be denoted by $\F$. It is defined by 
$$\F(\phi)(\theta)=\sum_{ n \in \Z} f(n) e^{2 \pi i n \theta},~~~\textrm{for~~~} \phi \in \ell^2(\Z) \textrm{~~and~~} \theta \in \T.$$
We need also the following lemma.
\begin{lem}\label{BW}Let $\phi \in\ell^2(\Z)$ and $\varphi \in \mathbb{A}(\T)$. Suppose  that 
$\varphi$ is differentiable with $\varphi' \in L^2(\T)$ and  $\widehat{\varphi}(0)=0.$ Then
\begin{align}
\|\F^{-1} (\varphi \F(\phi))\|_{\infty} \leq \|\varphi\|_\mathbb{A} \big\|f\big\|_{\infty}, \label{WB1}\\
\|\varphi\|_\mathbb{A} \leq C \sqrt{\big\|\varphi\big\|_{2}} \sqrt{\big\|\varphi'\big\|_{2}}
\label{WB2}.
\end{align} 
\end{lem}
\begin{proof}By the convolution theorem, we have
$$\F^{-1} (\varphi \F(\phi))=\F^{-1} (\varphi)*\phi.$$
Therefore
\begin{align*}
\|\F^{-1}(\varphi \F(\phi))\|_{\infty} 
&=\sup_{k \in \Z}\Big|\sum_{ n \in \Z} (\F^{-1}(\varphi))(n)\phi(k-n)\Big|\\
&\leq \Big|\sum_{ n \in \Z} (\F^{-1}(\varphi))(n)\Big| \big\|\phi\big\|_{\infty}\\
& \leq \|\varphi\|_\mathbb{A} \big\|\phi\big\|_{\infty}.
\end{align*}
We conclude the proof by noticing that $\F^{-1}(\varphi)(n)=\widehat{\varphi}(n)$, for all $n \in \Z$. We proceed now to the proof of \eqref{WB2}. By the time differential property of Fourier transform, we have
$$\widehat{\varphi}(n)=\frac{1}{in} \widehat{\varphi'}(n).$$
We refer to \cite[Theorem 1.6, p. 4]{Katznelson}, for the proof. 
Applying Cauchy-Schwarz inequality combined with Parseval equality, we get
\begin{align*}
\|\varphi\|_\mathbb{A} &=\sum_{ n \neq 0} \frac{1}{|n|} \big|\widehat{\varphi'}(n)\big|\\
&\leq  \Big(\sum_{ n \neq 0} \frac{1}{n^2}\Big)^{\frac{1}{2}} \Big(\sum_{ n \neq 0} \frac{1}{n^2} \big|\widehat{\varphi'}(n)\big|^2\Big)^{\frac{1}{2}}\\
& \leq \frac{\pi}{\sqrt{3}}  \Big(\sum_{ n \neq 0} \frac{1}{|n|} \big|\widehat{\varphi'}(n)\big| \big|\widehat{\varphi'}(n)\big|\Big)^{\frac{1}{2}}\\
&  \leq \frac{\pi}{\sqrt{3}}  \Big(\sum_{ n \neq 0} \frac{1}{n^2} \big|\widehat{\varphi'}(n)\big|^2\Big)^{\frac{1}{2}} 
\Big(\sum_{ n \neq 0} \big|\widehat{\varphi'}(n)\big|^2\Big)^{\frac{1}{2}}\\
&  \leq \frac{\pi}{\sqrt{3}} \sqrt{\big\|\varphi\big\|_{2}} \sqrt{\big\|\varphi'\big\|_{2}}
\end{align*}
\noindent This finish the proof of \eqref{WB2} and the lemma.
\end{proof}
\noindent Lemma \ref{BW} is stated in \cite{Bourgain-D} as equations (3.1) and (3.3).
 
\noindent At this point, we state the key lemma from \cite{Bourgain-D} in the proof of our main result.  

\begin{lem}\label{Bourgain-F}
Let $N_0,N_1,J$ be a positive integers such that $J >> N_1>N_0.$ Let $\phi, \psi \in \ell^{\infty}(J)$, and $0<\delta<1$ be a dyadic number. Suppose that for any interval $I \subset [0,J],$ with $|I|>N_0,$ we have 
		\begin{align}\label{C-delta}
		\sup_{ \theta \in [0,1) }\Bigg|\frac{1}{|I|}\sum_{n \in I} e^{2 \pi i n \theta} \psi(n)\Bigg|< \delta.
		\end{align}	
		Then, 
		\begin{align}\label{delta-inq}
		\frac{1}{J}\Bigg|\Big\{x \in [1,J]~~:~~ N_x (\textrm{dyadic}) \in [N_0,N_1] , N_x > \frac{1}{\delta},\notag \\
		\Big|\int \widehat{\phi}(\theta) \Big(P_{x,N_x}(\theta)\sigma_\delta\big(\big|P_{x,N_x}(\theta)\big|\big)\Big) e^{4  \pi i x \theta} d\theta\Big| >\delta^{\frac{1}{10^9}}\Big\}\Bigg|< C\delta^{\frac{1}{10^6}}.
		\end{align}
\end{lem}

\noindent It follows, by  Borel-Cantelli lemma, that under the condition \eqref{C-delta}, we have
$$\frac{1}{J}\sum_{x=1}^{J}m_{N_0,N_1}(\phi,\psi)<\delta^{\frac{1}{10^9}} \tend{\delta}{0} 0.$$
For the link between \eqref{delta-inq} and $m_{N_0,N_1}(\phi,\psi)$, we refer to the equation \eqref{Bconvolution}.\\

\noindent{}Lemma \ref{Bourgain-F} is a compilation of the principal result proved by Bourgain in \cite{Bourgain-D}. The assumption \eqref{C-delta} is stated as \cite[eq (2.10), page 142]{Bourgain-D} and the conclusion is formulated as the last equation \cite[eq (2.10), page 161]{Bourgain-D}. For the proof of Lemma \ref{Bourgain-F} one may need the following proposition.
\begin{prop}\label{BMZ}Let $J$ be a positive integer and $\delta>0$. Then,  there is a finite subset $E_0$ of the torus such that 
\begin{enumerate}
\item For any $(\lambda, \gamma) \in E_0^2,$ $\lambda\neq \gamma$, we have $$\big|\lambda-\gamma\big| \geq \frac{c}{J},$$ where $c$ is some absolute constant.
\item $|E_0| <C \delta^{-2},$ for some constant $C>0$. 
\item $\Big\{\lambda~~:~~ \ds \Big|\sum_{1 \leq n \leq J}\psi(n)e^{2\pi i n \lambda}\Big| > \delta J\Big\} \subset E_0\Big(\frac{c}{J}\Big),$ where $c$ is the constant in (1).
\end{enumerate} 
\end{prop}  
\begin{proof}Let us denote by $P_0$ the analytic polynomial given by 
	$$P_0(\lambda)=\sum_{1 \leq n \leq J}\psi(n)e^{2\pi i n \lambda}.$$
By, appealing to Marcinkiewicz-Zygmund interpolation inequalities  \cite[Theorem 7.5, Chapter X, p.28]{Zygmund}, if $\lambda_{J,j},j =0,\cdots,J$ are the $(J + 1)$-roots of the unity. Then, for
$\alpha > 1$ and for any  analytic 
polynomial $P$ of degree smaller or equal than $J$, we have 
\begin{align}\label{MZ}
\quad \quad \quad \quad \notag \\\frac{A_{\alpha}}{J}\sum_{j=0}^{J}\big|P(e^{2\pi i\lambda_{J,j}})\big|^{\alpha}
\leq \int_{\T}\Big|P(\lambda)\Big|^{\alpha} dz \leq \frac{B_{\alpha}}{J}\sum_{j=0}^{n-1}\big|P(e^{2\pi i\lambda_{J,j}})\big|^{\alpha},
\end{align}
where  $A_{\alpha}$ and $B_{\alpha}$ are independent of $J$ and $P$. According, a finite family of $(\lambda_{J,j})_{j=0,\cdots,J},$ is said to be a Marcinkiewicz-Zygmund family ({\it{MZ family}}) if \eqref{MZ} holds for it. Chui and Zhong \cite{CZ} proved that $(\lambda_{J,j})_{j=0,\cdots,J},$ is a MZ family if and only if the sample points  $(\lambda_{J,j})_{j=0,\cdots,J},$ satisfies 
$$\min_{0 \leq j<k \leq J}\Big|\lambda_{J,j}-\lambda_{J,k} \Big| \geq \frac{c}{J},$$
for some positive constant $c$. \\
Let $(\lambda_{J,j})_{j=0,\cdots,J},$ be a MZ family and $E_0$ be the maximal subset of it such that 
$$E_0 \subset \Big\{\lambda~~:~~ \ds \Big|P_0(\lambda)\Big| > \delta J\Big\}.$$ 
Then, obviously $E_0$ satisfy the condition $(1)$. To check that the condition $(2)$ is fulfilled, it suffices to notice that, by \eqref{MZ} for $\alpha=2$, we have 
\[ \delta^2 J^2|E_0|\leq \sum_{ j \in E_0}|P_0(\lambda_{J,j})||^2 \leq B_2 J \Big\|P_0\Big\|_2^2. \]

Moreover, by Parseval identity, $\big\|P_0\big\|_2^2 \leq J$ since $\psi$ is $1$-bounded. Whence, 
\[|E_0| \leq B_2 \delta^{-2}.\]
Finally, it is easy to check the last condition is satisfied by the maximality of $E_0$. Indeed, assume by contradiction that $(3)$ does not holds. Then, there is $\lambda$ such that $\lambda \in  \Big\{\lambda~~:~~ \ds \Big|P_0(\lambda)\Big| > \delta J\Big\}$ and 
$\big|\lambda-\lambda_{J,j}\big| \geq \frac{c}{J},$ for each $j \in E_0$. Therefore, either $\big|\lambda-\lambda_{J,j}\big| \geq \frac{c}{J},$ for each $j \in E_0^c$ and  we can thus add $\lambda$ to the family or for some for each $j \in E_0^c$, we have $\big|\lambda-\lambda_{J,j}\big| < \frac{c}{J}.$ In this case, we replace  $\lambda_{J,j} \in E_0^c$ by $\lambda$ in the family  $(\lambda_{J,j})_{j=0,\cdots,J}.$  In both cases we reach a contradiction since $E_0$ is maximal.     
\end{proof}

\noindent Proposition \ref{BMZ} is stated in \cite{Bourgain-D} as equation $(3.4)$ and $(6.1)$. The proof of Lemma \ref{Bourgain-F} is based also on the following 
Local Partition Lemma. Its proof is based on the nice properties of de la Vall\'ee-Poussin Kernel. We recall that de la Vall\'ee-Poussin Kernel $V_{(n,p)}$ of order $n+p$  is given by  

$$V_{(n,p)}(x)=\frac{n+p}{p}K_{n+p}(x)-\frac{n}{p}K_n(x),$$
where $K_n$ is Fej\`{e}r kernel. This later kernel is defined by
\begin{align}\label{FK}
K_n(x)=\sum_{j=-n}^{n}\Big(1-\frac{|j|}{n+1}\Big)e^{2\pi i j x}=\frac1{n+1}\Bigg\{ \frac{\sin\big( \pi (n+1) x\big)}
{\sin(\pi x)} \Bigg\}^2.
\end{align}\label{FVP}
Therefore 
\begin{align}\label{VVP}
	V_{(n,p)}(x)=\frac{1}{p}\frac{\big(\sin\big(\pi (n+p)x\big)\big)^2-\big(\sin\big(\pi n x\big)\big)^2}{(\sin(\pi x))^2}.
\end{align}
We further have $\widehat{V_{n,p}}(j)=1$, if $|j| \leq n$. 
We can associate to $V_{n,p}$ a compactly supported piecewise linear function $v_{n,p} $ on $\R$ given by 
\[v_{n,p}(t)=\begin{cases}
1 & \text{if } |t| \leq n\\
\frac{n+p-|t|}{p}  & \text{if } n \leq |t| \leq n+p\\
0 & \text{if } |t| \geq n+p
\end{cases}\]
The connection between those two families of functions is given by 
\begin{align}
	V_{(n,p)}(x)&=V_{(r m,(s-r) m)}(x)\\
	&= \sum_{j=-sm+1}^{sm-1} v_{r,s-r}\Big(\frac{j}{N}\Big)e^{2 \pi i j x},
\end{align}
where $m=\gcd(n,n+p)$,  $n= rm$ and $n+p=sm$.
 Notice that the classical de la Vall\'ee-Poussin kernel is given by 
\begin{align}
	V_{(n,n)}(x)&=2K_{2n}(x)-K_n(x)\\
	&=\sum_{ |j| \leq n}e^{2\pi i jx}+\sum_{ n <|j|< 2n}\Big(2-\frac{|j|}{N}\Big)e^{2\pi i jx}.
\end{align}
Moreover, as suggested by Bourgain in \cite[eq. (5.5), p.149]{Bourgain-D}, for any $\gamma>0$ and a family of the de la Vall\'ee-Poussin $V_{(n, \lfloor\gamma. n \rfloor)}$ kernel, the family of functions $v_{r,s-r}$ can be made smooth such that for any $M>\gamma^{-1}$, there is $D>1$ such that 
\begin{align}\label{SmoothB}
	\int_{\big\{|x| > \frac{M}{2.n+1}\big\}} | V_{(n, \lfloor \gamma. n \rfloor)}(x)| dx < C \gamma^{-1}\big(\gamma.M\big)^{-D}.
\end{align}
In connection with Bourgain's observation we have the following lemma.
\begin{lem}Let $n,D$ be a positive integers with $D>2$, $\gamma>0$ and $M >\gamma^{-1}$. Let $I$ be a finite interval of integers and $V_{(n,p)}$ be the associated de la Vall\'ee-Poussin kernel with $p > \gamma^{D+1}M^{D-2}|I|^2$ and $\textrm{supp}(v_{n,p}) \subset I$.  Then,  we have
\begin{align}
	\int_{\big\{|x| > \frac{M}{|I|} \big\}} | V_{(n,p)}(x)| dx < C \gamma^{-1}\big(\gamma.M\big)^{-D},
\end{align}
where $C$ is an absolute constant.
\end{lem}
\begin{proof}By \eqref{VVP}, we have 
\begin{align}\label{L1F1}
&\int_{\big\{|x| > \frac{M}{|I|} \big\}} | V_{(n,p)}| dx\\
	&\leq 
\int_{\big\{|x| > \frac{M}{|I|} \big\}} \frac{1}{p}\frac{\Big|\big(\sin\big(\pi (n+p)x\big)\big)^2-\big(\sin\big(\pi n x\big)\big)^2\Big|}{(\sin(\pi x))^2} dx 
\notag\\
&\leq \int_{\big\{|x| > \frac{M}{|I|} \big\}} \frac{1}{p}\frac{2}{(\sin(\pi x))^2} dx
\end{align}	
Now, remembering that $\sin(\frac{\tau}2)>\frac{\tau}{\pi}$ for any $0<\tau<\pi$, we obtain
\begin{align}\label{L1F2}
\int_{\big\{|x| > \frac{M}{|I|} \big\}} | V_{(n,p)}(x)| dx	&\leq 
\frac{1}{p}.\frac{1}{\frac{2.M^2}{|I|^2}} \\
&\leq C \gamma^{-1}\big(\gamma.M\big)^{-D}.
\end{align}
The last inequality is due to the fact that $p > \gamma^{D+1}M^{D-2}|I|^2$, and the proof of the lemma is complete.
\end{proof}

We state now the Local Partition Lemma of Bourgain and we present its proof. 
\begin{lem}[Local Partition Lemma of Bourgain (LP)]\label{Partition}  Let $I$ be a finite interval of integers and $P(x)$ be a trigonometric polynomial such that,
	$P(x)=\sum_{ n \in I}g_n e^{2 \pi i n x}$ with $|g_n| \leq 1$, for each $n \in I.$  Let $\Big\{I_{\alpha}\Big\}_{\alpha \in A}$ be a partition of $I$ in subintervales. Put
	$$P_{\alpha}(x) =\sum_{ n  \in I_{\alpha}}g_n e^{2 \pi i n x},$$
and let $E$ be a finite set in the torus identify with $(0,1]$. Then, for any $R>1$, we have
\begin{align}\label{B-partition}
&\int_{E(|I|^{-1})}\big|P(x)\big|^2 dx \notag 
\\
&\leq \sum_{ \alpha} \int_{E(R.|I_{\alpha}|^{-1})}\big|P_{\alpha}(x)\big|^2 dx+CR^{-D}|E||I|+CR^{-\frac14}|I|,
\end{align}
where $D$ is an arbitrarily chosen exponent and the constant $C$ depends on $D$.
\end{lem}
\begin{proof}By Riesz representation theorem,  there exist $\phi \in L^2(\T)$ such that
\begin{enumerate}
\item \label{R1}$\big\|\phi\big\|_2=1,$
\item \label{R2}$\textrm{supp}(\phi) \subset E(|I|^{-1})$,
\item \label{R3} $\ds \int_{E(|I|^{-1})} P(x) \overline{\phi}(x) dx=\Big\|P|_{E(|I|^{-1})}\Big\|_2.$
\end{enumerate}
Let $(V_\alpha)_{\alpha \in A}$ be a family of the smooth de la Vall\'ee-Poussin kernel which satisfy for each $\alpha \in A$, 
\begin{enumerate}
	\item $\textrm{supp}(\widehat{V_\alpha}) \subset I_\alpha,$
 	\item \eqref{SmoothB} with $\gamma \in (0,\frac{1}{10}).$
\end{enumerate} 
Put
$$Q_\alpha=P_\alpha*V_\alpha \textrm{~~~~~~~and~~~~~~~~~} \psi_\alpha=\phi_\alpha*V_\alpha,$$
where $\phi_\alpha=\F(\F^{-1}(\phi)|_{I_\alpha}).$ Notice that we have $\psi_\alpha=\phi*V_\alpha,$  
$\big\|\phi_\alpha\big\|_2 \leq 1,$ and by the convolution theorem $\big\|\psi_\alpha\big\|_2 \leq \big\|\phi\|_2=1.$
Moreover, by proprieties \ref{R3} of $\phi$ and the definition of $P_\alpha$ and $\phi_\alpha$, we have
\begin{align}\label{PB1}
\Big\|P|_{E(|I|^{-1})}\Big\|_2&=	<P(x), \phi> \notag\\
&=\sum_{ \alpha \in A} \sum_{ n \in I_\alpha} g_n \widehat{\phi}(n) \notag\\
&= \sum_{ \alpha \in A} \sum_{n \in \Z} g_n \widehat{\phi}(n) \mathbbm{1}_{I_\alpha}(n) \notag\\
&=\sum_{ \alpha \in A} \sum_ {n \in  \Z} \widehat{P_\alpha}(n) \widehat{\phi}|_{|I_\alpha}(n) \notag\\
&=\sum_{ \alpha \in A}<P_\alpha,\phi_\alpha> \notag\\
&\leq \sum_{ \alpha \in A}\Big|<P_\alpha,\phi_\alpha>\Big|
\end{align}

For the estimation of $\Big\|P|_{E(|I|^{-1})}\Big\|_2$, we are going to replace up to some errors $\Big|<P_\alpha,\phi_\alpha>\Big|$ with $\Big|<Q_\alpha,\phi_\alpha>\Big|$. We start with the estimation of the error.

Applying Parserval identity combined with the definition and the properties of $P_\alpha$ and $Q_\alpha$, we get  
\begin{align}\label{PB2}
\Big\|P_\alpha-Q_\alpha\Big\|_2^2&=\Big\|P_\alpha-P_\alpha*V_\alpha\Big\|_2^2 \notag\\
&=\sum_{ n \in I_\alpha}\Big|\widehat{\Big(P_\alpha-P_\alpha*V_\alpha\Big)}(n)\Big|^2 \notag \\
&=\sum_{ n \in I_\alpha}\Big|\widehat{P_\alpha}(n)\big(1-\widehat{V_\alpha}(n)\big)\Big|^2
\end{align}
But, by our assumption, we can write $I_\alpha= I_\alpha^1 \bigcup I_\alpha^2,$ with $I_\alpha^1 \bigcap I_\alpha=\emptyset$ and $|I_\alpha^2|=\gamma. |I_\alpha|$. This combined with the nice proprieties of $V_\alpha$  and the $1$-boundness of $(g_n)$ gives
\begin{align}\label{PB3}
\Big\|P_\alpha-Q_\alpha\Big\|_2^2
&=\sum_{ n \in I_\alpha^{2}}\Big|\widehat{P_\alpha}(n)\big(1-\widehat{V_\alpha}(n)\big)\Big|^2\\
&< 4 \gamma |I_\alpha|.
\end{align}
Whence 
\begin{align}\label{PB4}
\Big\|P_\alpha-Q_\alpha\Big\|_2  < 2 \sqrt{\gamma |I_\alpha|}.
\end{align}
Applying Cauchy-Schwarz inequality and taking into account the $1$-boundness of $\|\phi_\alpha\|_2$, we obtain
\begin{align}\label{PB-error}
\Big|\langle P_\alpha,\phi_\alpha\rangle-\langle Q_\alpha,\phi_\alpha\rangle\Big| &\leq \Big\|P_\alpha-Q_\alpha\Big\|_2 \big\|\phi_\alpha\big\|_2 \notag\\
& \leq 2 \sqrt{\gamma |I_\alpha|}.
\end{align}

Now, we proceed to the estimation of $\langle Q_\alpha,\phi_\alpha\rangle$.  By the definition of $Q_\alpha, \phi_\alpha$ and $\psi_\alpha$, we have
\begin{align}\label{PB5}
\langle Q_\alpha,\phi_\alpha\rangle &=\langle P_\alpha*V_\alpha,\phi_\alpha\rangle\notag \\
&=\sum_{ n \in I_\alpha} \widehat{P_\alpha}(n) \widehat{V_\alpha}(n) \overline{\widehat{\phi_\alpha}(n)}\notag \\
&= \langle P_\alpha,\phi_\alpha*V_\alpha\rangle \notag \\
&=\langle P_\alpha,\psi_\alpha\rangle. 
\end{align}
We can thus write 
\begin{align}\label{PB6}
\langle Q_\alpha,\phi_\alpha\rangle&=\langle P_\alpha (\1_{E(|I|^{-1})+\1_{E(|I|^{-1})^c}}),\psi_\alpha \rangle \notag \\
 &= \langle P_\alpha \1_{E(|I|^{-1})},\psi_\alpha \rangle + \langle P_\alpha \1_{E(|I|^{-1})^c},\psi_\alpha \rangle 
\end{align}
This combined with Cauchy-Schwarz inequality gives
\begin{align}\label{PB7}
\Big|\langle Q_\alpha,\phi_\alpha\rangle\Big| \leq \big\| P_\alpha \1_{E\big(\frac{R}{|I_\alpha|}\big)}\big\|_2 \big\|\psi_\alpha \big\|_2+ \big\| P_\alpha \big\|_2  \big\| \psi_\alpha\1_{E\big(\frac{R}{|I_\alpha|}\big)^c} \big\|_2.
\end{align}
Applying again Cauchy-Schwarz inequality combined with \eqref{R1} and \eqref{R2} we obtain, for any $\theta \in (0,1]$
\begin{align}\label{PB8}
\big|\psi(\theta)\big|&=\Big|\int_{E(|I|^{-1})}\phi(t) V_\alpha(\theta-t) dt\Big|\notag\\
&\leq \Big(\int_{E(|I|^{-1})}  \big| V_\alpha(\theta-t) \big|^2dt\Big)^{\frac12}\notag\\
&\leq  \big\| V_\alpha \big\|_2=1.
\end{align}
Therefore
\begin{align}\label{PB9}
\big\| \psi_\alpha\1_{E\big(\frac{R}{|I_\alpha|}\big)^c} \big\|_2^2&=\int \big|\psi_\alpha(\theta)\big|^2\1_{E\big(\frac{R}{|I_\alpha|}\big)^c}(\theta) d\theta\notag \\
&\leq  \int \big|\psi_\alpha(\theta)\big|\1_{E\big(\frac{R}{|I_\alpha|}\big)^c}(\theta) d\theta \notag \\
& \leq \int \1_{E\big(\frac{R}{|I_\alpha|}\big)^c}(\theta) \int_{E(|I|^{-1})}|\phi(t)||V_\alpha(\theta-t)| dt
\end{align} 
But, assuming $t \in E(|I|^{-1})$, we get that  $\theta \in E\big(\frac{R}{|I_\alpha|}\big)^c$ implies $|\theta-t| > \frac{R}{2 |I_\alpha|} .$ If not, since for some $\lambda \in E$, we have $|\lambda-t| <\frac{1}{|I|}< \frac{R}{2 |I_\alpha|}$, by the triangle inequality, we get $|\lambda-\theta|< \frac{R}{2 |I_\alpha|}$ which  gives a contradiction.  We can thus rewrite \eqref{PB8} as follows
\begin{align}\label{PB10}
\big\| \psi_\alpha\1_{E\big(\frac{R}{|I_\alpha|}\big)^c} \big\|_2^2
 \leq \int_{E(|I|^{-1})}|\phi(t)| \int_{\{|\theta-t| > \frac{R}{2 |I_\alpha|}\}} |V_\alpha(\theta-t)| dt
\end{align} 
Applying \eqref{SmoothB} and Cauchy-Schwarz inequality to get
\begin{align}\label{PB11}
\big\| \psi_\alpha\1_{E\big(\frac{R}{|I_\alpha|}\big)^c} \big\|_2^2
& \leq \int_{E(|I|^{-1})}|\phi(t)| \int_{\{|\theta-t| > \frac{R}{2 |I_\alpha|}\}} |V_\alpha(\theta-t)| dt \notag\\
&\leq \int_{E(|I|^{-1})}|\phi(t)| dt \Big(C. \gamma^{-1} (\gamma.R)^{-D}\Big) \notag \\
&\leq \Big(\text{Leb}(E(|I|^{-1}))\Big)^{\frac12} \Big(C. \gamma^{-1} (\gamma.R)^{-D}\Big) \notag\\
&\leq \Big(|E|.|I|^{-1}\Big)^{\frac12} \Big(C. \gamma^{-1} (\gamma.R)^{-D}\Big),
\end{align} 
where $\text{Leb}$ stand for Lebesgue measure on the circle.
Summarizing, we get
\begin{align}\label{PB-final1}
\Big\|P|_{E(|I|^{-1})}\Big\|_2 &\leq \sum_{ \alpha \in A} \Big|\langle Q_\alpha,\phi_\alpha\rangle\Big|+	2 \sqrt{\gamma} \sum_{ \alpha \in A}\sqrt{|I_\alpha|} \notag\\
&\leq  \sum_{ \alpha \in A}\big\| P_\alpha \1_{E\big(\frac{R}{|I_\alpha|}\big)}\big\|_2 \big\|\phi_\alpha\big\|_2\\&+ \Big(|E|.|I|^{-1}\Big)^{\frac12} \Big(C. \gamma^{-1} (\gamma.R)^{-D}\Big)^{\frac12} 
\sum_{ \alpha \in A} \big\|P_\alpha\big\|_2 \notag \\&+ 	2 \sqrt{\gamma} \sum_{ \alpha \in A}\sqrt{|I_\alpha|}\notag
\end{align}
by \eqref{PB1}, \eqref{PB-error},\eqref{PB7} and \eqref{PB11}. Applying again Cauchy-Schwarz inequality we see that 
\begin{align*}
	\sum_{ \alpha \in A}\sqrt{|I_\alpha|} &\leq \sqrt{|I|}\\
	\sum_{ \alpha \in A}\big\| P_\alpha \1_{E\big(\frac{R}{|I_\alpha|}\big)}\big\|_2 \big\|\phi_\alpha\big\|_2
	&\leq \Big(\sum_{ \alpha \in A}\big\| P_\alpha \1_{E\big(\frac{R}{|I_\alpha|}\big)} \big\|^2\Big)^{\frac12} \\	
	\sum_{ \alpha \in A} \big\|P_\alpha\big\|_2 &\leq  \sqrt{|I|},
\end{align*}
since, by our assumption, $\ds \sum_{ \alpha \in A}\big\|\phi_\alpha\big\|_2^2 \leq 1$ and $\big\|P_\alpha\big\|_2^2 \leq|I_\alpha|,$ for each $\alpha \in A$. Consequently, we can rewrite \eqref{PB-final1} as follows 
\begin{align}\label{PB-final2}
\Big\|P|_{E(|I|^{-1})}\Big\|_2 &\leq \sum_{ \alpha \in A} \Big|\langle Q_\alpha,\phi_\alpha\rangle\Big|+	2 \sqrt{\gamma} \sum_{ \alpha \in A}\sqrt{|I_\alpha|} \notag\\
&\leq  \Big(\sum_{ \alpha \in A}\big\| P_\alpha \1_{E\big(\frac{R}{|I_\alpha|}\big)}\big\|_2^2 \Big)^{\frac12}\\&+ \sqrt{|E|} \Big(C. \gamma^{-1} (\gamma.R)^{-D}\Big)^{\frac12} 
 + 	2 \sqrt{\gamma} \sqrt{|I|}\notag
\end{align}
Taking $\gamma=\frac{1}{\sqrt[4]{R}}$, we obtain
\begin{align}\label{PB-final3}
\Big\|P|_{E(|I|^{-1})}\Big\|_2 
&\leq  \Big(\sum_{ \alpha \in A}\big\| P_\alpha \1_{E\big(\frac{R}{|I_\alpha|}\big)}\big\|_2^2 \Big)^{\frac12}\\&+ \sqrt{|E|} \Big(\sqrt{C}.R^{\frac{1-3D}{8}}\Big) \notag
+ 	2 R^{-\frac18} \sqrt{|I|}\notag
\end{align}
Squaring, we get 
\begin{align}\label{PB-final4}
\Big\|P|_{E(|I|^{-1})}\Big\|_2^2 
&\leq \sum_{ \alpha \in A}\big\| P_\alpha \1_{E\big(\frac{R}{|I_\alpha|}\big)}\big\|_2^2+C |E|  R^{\frac{1-3D}{4}} 
+ 	4 R^{-\frac14} |I|\\&+2 \Big(\sum_{ \alpha \in A}\big\| P_\alpha \1_{E\big(\frac{R}{|I_\alpha|}\big)}\big\|_2^2 \Big)^{\frac12} \sqrt{|E|} \Big(\sqrt{C}.R^{\frac{1-3D}{4}}\Big)\notag\\
&+4 \Big(\sum_{ \alpha \in A}\big\| P_\alpha \1_{E\big(\frac{R}{|I_\alpha|}\big)}\big\|_2^2 \Big)^{\frac12}  R^{-\frac18} \sqrt{|I|} \notag\\
&+4\sqrt{|E|} \Big(\sqrt{C}.R^{\frac{1-3D}{8}}\Big) R^{-\frac18} \sqrt{|I|} \notag
\end{align}
by the virtue of the simple identity $$(a+b+c)^2=a^2+b^2+c^2+2ab+2ac+2bc,~~~ \forall a,b,c \in \R.$$
Now, remembering that 
$$\sum_{ \alpha \in A}\big\| P_\alpha \1_{E\big(\frac{R}{|I_\alpha|}\big)}\big\|_2^2  \leq 
\sum_{ \alpha \in A}\big\| P_\alpha \big\|_2^2 \leq \sum_{ \alpha \in A}\big| I_\alpha| \leq |I|.  $$
We can rewrite \eqref{PB-final4} as follows
\begin{align}\label{PB-final5}
\Big\|P|_{E(|I|^{-1})}\Big\|_2^2 
&\leq \sum_{ \alpha \in A}\big\| P_\alpha \1_{E\big(\frac{R}{|I_\alpha|}\big)}\big\|_2^2+C |E|  R^{\frac{1-3D}{4}} 
+ 	4 R^{-\frac14} |I|\\&+2 \sqrt{|I|} \sqrt{|E|} \Big(\sqrt{C}.R^{\frac{1-3D}{4}}\Big)\notag\\
&+4 \sqrt{|I|} R^{-\frac18} \sqrt{|I|} \notag\\
&+4\sqrt{|E|} \Big(\sqrt{C}.R^{\frac{1-3D}{8}}\Big) R^{-\frac18} \sqrt{|I|} \notag
\end{align}
We thus conclude that 
\begin{align}\label{PB-last}
\Big\|P|_{E(|I|^{-1})}\Big\|_2^2 
\leq \sum_{ \alpha \in A}\big\| P_\alpha \1_{E\big(\frac{R}{|I_\alpha|}\big)}\big\|_2^2+K R^{-D'}  |E||I| 
+ 	8 R^{-\frac14}|I|
\end{align}
for some absolutely positive constant $K$ and an arbitrary  $D'>0$ since $D>1$ is arbitrary and $D'=\frac{1-3D}{8}$. This complete the proof of the lemma.
\end{proof}
\begin{rem}Despite the fact that our proof follows Bourgain's proof, it is slightly different.
\end{rem}
\noindent The LP Lemma of Bourgain is useful in the following form.
\begin{lem}[The $\epsilon$-localization of LP Lemma]\label{PartitionII}  Let $I$ be a finite interval of integers and $P(x)$ be a trigonometric polynomial such that,
	$P(x)=\sum_{ n \in I}g_n e^{2 \pi i n x}$ with $|g_n| \leq 1$, for each $n \in I.$  Let $\Big\{I_{\alpha}\Big\}_{\alpha \in A}$ be a partition of $I$ in subintervales such that $|I_\alpha|< \epsilon |I|$ with $0<\epsilon <\frac{1}{10}$. Put
	$$P_{\alpha}(x) =\sum_{ n  \in I_{\alpha}}g_n e^{2 \pi i n x},$$
	and let $E$ be a finite set in the torus identify with $(0,1]$. Then, for any $R>1$, we have
	\begin{align}\label{B-partitionII}
	&\int_{E(R|I|^{-1})}\big|P(x)\big|^2 dx \notag 
	\\
	&\leq \sum_{ \alpha} \int_{E(R.|I_{\alpha}|^{-1})}\big|P_{\alpha}(x)\big|^2 dx+CR^{-D}|E||I|+CR^{-\frac14}|I|,
	\end{align}
	where $D$ is an arbitrarily chosen exponent and the constant $C$ depends on $D$.
\end{lem}
\begin{proof}We are going to apply Lemma \ref{Partition}. For that let $E'$ be a subset of $E\big(\frac{R}{|I|}\big)$ such that $|E'|<2 R|E|$ and $E\big(\frac{R}{|I|}\big) \subset E'\big(\frac{1}{|I|}\big)$. We thus have, for each $\alpha \in A$, 
	$$ E'\Big(\frac{R}{2|I_\alpha|}\Big) \subset E\Big(\frac{R}{|I_\alpha|}\Big).$$
Indeed, if $\lambda \in E'\big(\frac{R}{2|I_\alpha|}\big)$, then, for some $\xi \in E'$, we have 
$$\Big|\lambda-\xi\Big| < \frac{R}{2|I_\alpha|}.$$
But $E'$ is a a subset of $E\big(\frac{R}{|I|}\big)$. Therefore, for some $\eta \in E$, we have 
$$ \Big|\xi-\eta\Big| <\frac{R}{|I|}<\frac{R}{2|I_\alpha|}.$$
Applying the triangle inequality, we get 
$$\Big|\lambda-\eta\Big| < \frac{R}{|I_\alpha|}.$$  
Now, by appealing  to Lemma \ref{Partition}, we obtain
\begin{align*}
\Big\|P|_{E'(|I|^{-1})}\Big\|_2^2 
\leq \sum_{ \alpha \in A}\big\| P_\alpha \1_{E'\big(\frac{R}{2|I_\alpha|}\big)}\big\|_2^2+K \Big(\frac{R}{2}\Big)^{-D}  |E'||I| 
+ 	8 \Big(\frac{R}{2}\Big)^{-\frac14}|I|	
\end{align*}
Whence
\begin{align*}
\Big\|P|_{E\big(\frac{R}{|I|}\big)}\Big\|_2^2 
\leq \sum_{ \alpha \in A}\big\| P_\alpha \1_{E\big(\frac{R}{|I_\alpha|}\big)}\big\|_2^2+C R^{-D}  |E||I| 
+ 	16 R^{-\frac14}|I|,	
\end{align*}
and this achieve the proof of the lemma.
\end{proof}
The proof of the fundamental lemma (Lemma \ref{Bourgain-F}) is based also on the famous $\lambda$-separated lemmas due to Bourgain. These lemmas are in the heart of Bourgain's method.  Its proof involve  L\'epingle inequalities (also called L\'epingle Lemma 
). We refer to \cite[Lemma 3.3]{Bourgain-IHP} (see also \cite{Thouvenot}). For this later inequalities, one need to introduce the variation norms. Let $(a_n)_{n \in \N}$ be a complex sequence and $s_0$ a positive integer. The variation norm of order $s$ is given by
$$\big\|(a_n)\|_{v_s}=\sup_{ \overset{J}{n_1<n_2<\cdots<n_{J}}}\Big(\sum_{k=1}^{J-1}|a_{n_{k+1}}-a_{n_k}|^s\Big)^{\frac1{s}}.$$

\noindent At this point, we are able to recall the $\lambda$-separated Lemmas.

\begin{lem}[$\lambda$-separated Lemma]\label{lambda-B}Let $\lambda_1,\cdots,\lambda_K$ be a $K$ points on the circle such that $|\lambda_i-\lambda_j| \geq \frac{1}{2^{s-1}}$, $\forall i \neq j$, with $s>0$. Let $f \in L^2[0,1)$. Then, there is an absolute constant $C>0$ such that 
	$$\Bigg\|\sup_{ j >s }\Big|\int_{V_j}e^{2\pi in \alpha} f(\alpha)d\alpha|\Bigg\|_{\ell^2(\Z)}
	\leq C (\log(K))^2 \big\|f\big\|_2,$$
	where $V_j$ is $\frac{1}{2^j}$-neighborhood of $\{\lambda_1,\cdots,\lambda_K\}$ given by
	$$V_j=\Big\{\lambda \in [0,1)/ \min_{1 \leq r \leq K} |\lambda-\lambda_r|<\frac{1}{2^{j}}\Big\}.$$  
\end{lem}
As a consequence of L\'epingle Lemma, we have the following lemma needed here \cite[Lemma 3.23]{Bourgain-D}.
\begin{lem}[The entropic $\lambda$-separated Lemma]Let $\lambda_1,\cdots,\lambda_K$ be a $K$ points on the circle such that $|\lambda_i-\lambda_j| \geq \tau$, $\forall i \neq j$. Let $f \in \ell^2(\Z))$, $\chi=\1_{[0,1]}$ and 
	$\chi_N=\frac{1}{N}\1_{[0,N]}$ and consider, for $x \in \Z$,
\begin{align}
	\Gamma_x=\Big\{\int_{0}^{1}\widehat{f}(\lambda) \widehat{\chi_N}(\lambda-\lambda_k\big)_{1 \leq k \leq K}, N (\textrm{dyadic}) > \frac{1}{\tau}\Big\}
\end{align}
as a subset of $K$-dimensional Hilbert space. Then, for any $t>0$
\begin{align}
\int N(\Gamma_x,t)dx \leq\frac{c}{t^2}\big\|f\big\|_2^2,
\end{align}
where, for a subset $A$ of a Hilbert space $H$, $N(A,\epsilon), \epsilon>0$ stand for the metrical entropy numbers, that is, the minimal number of balls of radius $\epsilon$ needed to cover $A$. 
\end{lem} 

More details on Lemma \ref{Bourgain-F} will be given in the revised version of \cite{elabdal-poly}.

\noindent We proceed now to the proof of our main result (Theorem \ref{Mainofmain}).
\begin{proof}[\textbf{Proof of Theorem \ref{Mainofmain}}.]Without loss of generality, we assume that the map $T$ is  ergodic. We further assume  that $f,g$ are in $L^\infty(X,\mu)$ with $\big|\big\|g\big\|\big|_{U^2}=0$ (notice that we can interchange the role of $f$ and $g$). Therefore, by Lemma \ref{WW-Gowers} combined with Proposition \ref{Assani-Moore}, for any prime $p \neq q$, we have 
\begin{align}\label{Key-conv}
\frac{1}{N}\sum_{n=1}^{N}z^{(p-q)n} g(T^{-pn}x)g(T^{-qn}x) \tend{N}{+\infty}0, \textrm{~~for~~almost~~all~~}x
\end{align}
We thus get, by Proposition \ref{WWKBSZ}, the following
\begin{align}\label{CM-delta1}
\frac{1}{N}\sum_{n=1}^{N}\bmnu(n) z^{-n}g(T^{-n}x) \tend{N}{+\infty}0, \textrm{~~for~~almost~~all~~}x.
\end{align}

It follows that if $f$ is an eigenfunction, then the convergence  holds.

Now, for $f$ and $g$ in the orthocompelment of eigenfunctions, following Bourgain's approach, we will use the finitary method. Therefore,   by \eqref{CM-delta1}, we can write 
\begin{align}\label{CM-delta2}
\frac{1}{N}\sum_{x-N\leq n<x}\bmnu_x(n) z^{-n}\psi(n) \tend{N}{+\infty}0, 
\end{align}
where $\psi \in \ell^{\infty}(\Z)$ with $\||\psi\||_{U^2}=0.$ We set
$$Q_{x,N,\bmnu}(\theta)=\frac{1}{N}\sum_{x-N\leq n<x}\bmnu_x(n) e^{-2 \pi i n \theta}\psi(n).$$

Therefore, for $\delta_{0}>0,$ there exist $N_0 \in \N$ such that for any interval $I \subset [1,J)$ with $|I|>N_0$ we have
\begin{align}\label{Bdelta}
\Big|\frac{1}{|I|}\sum_{ n \in I}\bmnu_{x_0}(n) e^{-2 \pi i n \theta}\psi(n)\Big|<\delta_0,
\end{align}
where $x_0=\min(I).$ We recall that $\bmnu_{x}$ is defined by 
$$\bmnu_{x}(n)=\bmnu(x-n).$$
Of-course, we extend $\bmnu$ to negative integer in the usual fashion.\\

\noindent We proceed now to the application of the fundamental lemma of Bourgain (Lemma \ref{Bourgain-F}). We stress that this is not a direct application of the lemma. In fact, one may need to check that the condition ((3.6) to (3.9)) of Lemma 3.5 from \cite{Bourgain-D} are satisfied. For that and the convenience of the reader, we recall first Lemma 3.5 from  \cite{Bourgain-D}.
\begin{lem}[Lemma 3.5 in \cite{Bourgain-D}]
	Let $0<\delta<1$, $N>1$ and $\Omega$ on $\Z \times [0,1)$ satisfying the condition
	\begin{enumerate}
		\item \label{BC1}$|\Omega|<\delta,$ 
		\item \label{BC2} $\big\|\Omega(x,.)\big\|_2 <N^{-1/2}.$
		\item \label{BC3} $\big\|\Omega(x,.)-\Omega(x',.)\big\|_2 < \sqrt{\frac{|x-x'|}{N}}.$ \footnote{In \cite{Bourgain-D}, there is a misprint and the argument used in the proof apply \eqref{BC3}.}	
		\item \label{BC4} $\big\|\partial_{\theta}\Big(e^{2 \pi i x \theta}\Omega(x,\theta)\Big)\big\|_2 <N^{1/2}.$
	\end{enumerate} 
	Let $0<\eta<1$ be a function on $[0,1)$ satisfying 
	$$|\eta'(\lambda)|< K.N.$$ 
	Let $f$ be $1$-bounded finitely supported function on $\Z$  and define
	$$F(x)=\ds \int \widehat{f}(\lambda)\Omega(x,\lambda) \eta(\lambda) e^{4 \pi i x \lambda} d\lambda,$$
	Let $t> K^{1/2}.N^{-1/4}$.  There is the estimate 
	\begin{align}
	\Big|\Big\{x \in \Z | |F(x)|>t\Big \}\Big| <c K^3 t^{-8} \delta^2 \int |\widehat{f}(\lambda)|^2 |\eta(\lambda)|^2 d\lambda.
	\end{align}
\end{lem}

Let us now notice that, by \eqref{CM-delta2}, the condition \eqref{BC1} is fulfilled, and a straightforward computation yields that $\big\|Q_{x,N,\bmnu}\|_2$ is less than $N^{-1/2}$ since $\bmnu$ and $\psi$ are $1$-bounded. We further have that the condition \eqref{BC2} and \eqref{BC3} are fulfilled, since  the map $z \mapsto z \sigma_\delta(|z|)$ is a Lipschitz function and the constant is independent of $\delta$. Indeed,  
\begin{align}\label{eq38}
	&\Big\|e^{2\pi i x \theta} Q_{x,N,\bmnu}(\theta)-e^{2\pi i x'\theta}Q_{x',N,\bmnu}(\theta)\Big\|_2^2\notag\\
	&=\Big\|\frac{1}{N}\sum_{n=1}^{N}\bmnu(n)\psi(x-n)e^{2 \pi i n\theta }-\frac{1}{N}\sum_{n=1}^{N}\bmnu(n)\psi(x'-n)e^{2 \pi i n \theta}\Big\|_2^2 \notag\\
	&=\frac{1}{N^2}\sum_{n=1}^{N}|\bmnu(n)|^2 \big|\psi(x-n)-\psi(x'-n)\big|^2\\		
	&\leq \frac{1}{N^2}\sum_{n=1}^{N} \big|\psi(x-n)-\psi(x'-n)\big|^2 \label{key-b-un}\\ &=\Big\|e^{2\pi i x} Q_{x,N,\bun}(\theta)-e^{2\pi i x'}Q_{x',N,\bun}(\theta)\Big\|_2^2\\\notag
\end{align}
The  inequality \eqref{key-b-un} follows from the $1$-boundness of $\bmnu$.  We further have
\begin{align}\label{eq38bis}
&\frac{1}{N^2}\sum_{n=1}^{N} \big|\psi(x-n)-\psi(x'-n)\big|^2\notag \\
&=\Big\|\frac{1}{N}\sum_{n=1}^{N}\psi(x-n)e^{2 \pi i n\theta}-\frac{1}{N}\sum_{n=1}^{N}\psi(x'-n)e^{2 \pi i n \theta}\Big\|_2^2
\notag \\
&=\Big\|\frac{1}{N}\sum_{n=1}^{N}\psi(n)e^{2 \pi i n\theta}-\frac{1}{N}\sum_{n=1}^{N}\psi(x'-n)e^{2 \pi i n \theta}\Big\|_2^2
\notag \\
&=\Big\|e^{2\pi i x} Q_{x,N,\bun}(\theta)-e^{2\pi i x'}Q_{x',N,\bun}(\theta)\Big\|_2^2.
\end{align}
Now, we write  
\begin{align}\label{eq39bis}
&\Big\|e^{2\pi i x} Q_{x,N,\bun}(\theta)-e^{2\pi i x'}Q_{x',N,\bun}(\theta)\Big\|_2
\\
&= \Big(e^{2\pi i x\theta}-e^{2\pi i x'\theta}\Big) Q_{x,N,\bun}(\theta)+e^{2\pi i x'} \Big(Q_{x,N,\bun}(\theta)-Q_{x',N,\bun}(\theta)\Big)\Big\|_2 \\
&\leq I+II,
\end{align}
where
$$I=\Big\|\Big(e^{2\pi i x\theta}-e^{2\pi i x'\theta}\Big) Q_{x,N,\bun}(\theta)\Big\|_2,$$
and 
$$II= \Big\|e^{2\pi i x'} \Big(Q_{x,N,\bun}(\theta)-Q_{x',N,\bun}(\theta)\Big) \Big\|_2$$
Applying the following classical inequality
$$\Big|\frac{e^{2\pi i x}-e^{2\pi i x'}}{2}\Big|\leq 
 \Big|\frac{e^{2\pi i x}-e^{2\pi i x'}}{2}\Big|^\frac12 \leq
 C|x-x'|^\frac{1}{2},
$$
we obtain
$$I^2 \leq C. \frac{|x-x'|}{N}.$$
We further have
\begin{align}\label{eq40bis} 
&\Big\|e^{2\pi i x'} \Big(Q_{x,N,\bun}(\theta)-Q_{x',N,\bun}(\theta)\Big) \Big\|_2^2\\
&=\Big\|\frac{1}{N}\sum_{x-N<n<x}\psi(n)e^{2 \pi i n}-\frac{1}{N}\sum_{x'-N<n<x'}\psi(n)e^{2 \pi i n}\Big\|_2^2
\\
&\leq \frac{|x-x'|}{N}.
\end{align}
At this we point out that for the application of Bourgain Lemma ((Lemma \ref{Bourgain-F})), we need to renormalize the polynomials $Q_{x,N,\bmnu}$ by the constant $(1+C)/2$.  \\
It is still to check that the condition \eqref{BC4} is fulfilled. This can be done, by applying Bernstein-Zygmund inequalities \cite[Theorem 3.13, Chapter X, p. 11]{Zygmund} combined with the $1$-boundness of $\bmnu$ and $\psi$. We thus get 
\begin{align}\label{eq39}
&\Big\|\partial_{\theta}\Big(e^{2 \pi i x \theta }Q_{x,N,\bun}(\theta)\Big)\Big\|_2
\leq N.\Big\|\frac{1}{N}\sum_{n=1}^{N}\bmnu(n)\psi(x-n)e^{2 \pi i n\theta }\Big\|_2
<\sqrt{N}
\end{align}

\noindent Following the path of Bourgain's  proof, the rest of the proof is accomplished by applying the LP lemma (Lemma \ref{Partition}) (see also \cite[Lemma 5.1]{Bourgain-D}) and Lemma \ref{lambda-B}  (see also \cite[Lemma 3.23]{Bourgain-D} and \cite[Lemma 4.1]{Bourgain-IHP}  )). We thus  get, for any $\phi \in \ell^\infty(\Z)$, 
\begin{align*}
&\frac{1}{J}\Bigg|\Big\{x \in [1,J]~~:~~ \notag 
\max_{\overset{N_0<N<N_1, }{N (\textrm{dyadic})>\frac{1}{\delta}}}\Big|\frac{1}{N}\sum_{n=1}^{N}\bmnu(n)\phi(n+x)\psi(n-x)\Big| >\delta^{\frac{1}{10^9}}\Big\}\Bigg|\\&< C\delta^{\frac{1}{10^6}}.
\end{align*}
Since, by Fourier transform transfer, we have
\begin{align*}
&\frac{1}{N}\sum_{n=1}^{N}\bmnu(n)\phi(n)\psi(n)\\
&=\int_{0}^{1} \widehat{\phi}(\theta) \Big(\frac{1}{N}\sum_{n=1}^{N} \bmnu(n)  e^{-2 \pi i (x-n) \theta}\psi(x-n)\Big) e^{4 i \pi x \theta} d\theta.
\end{align*}

Moreover, we have
\begin{align}\label{eq217}
&\Big|\int_{0}^{1} \widehat{\phi}(\theta) \Big(\frac{1}{N}\sum_{n=1}^{N} \bmnu(n)  e^{-2 \pi i (x-n) \theta}\psi(x-n)\Big) e^{4 i \pi x \theta} d\theta\Big|\notag \\
&\leq \sum_{ \overset{\delta_{0}>\delta> N_x^{-1}}{\textrm{dyadic}}}
\Big|\int_{0}^{1} \widehat{\phi}(\theta) \Big(Q_{x,N}(\theta)\sigma_\delta(\big|Q_{x,N}(\theta)\big|)\Big) \Big) e^{4 i \pi x \theta} d\theta\Big|+ \frac{C}{\sqrt[4]{N_x}} .
\end{align}
This is due to the following fact
\begin{align}\label{eq-B2to3}
&\Big|\int_{0}^{1} \widehat{\phi}(\theta) \Big(\frac{1}{N}\sum_{n=1}^{N} \bmnu(n)  e^{-2 \pi i (x-n) \theta}\psi(x-n)(\theta)\Big) e^{4 i \pi x \theta} d\theta\Big| \notag\\
&=\Big|\int_{0}^{1} \widehat{\phi}(\theta) \Big(\sum_{\overset{0<\delta<1}{\textrm{~~dyadic~~}}} Q_{x,N}(\theta)\sigma_\delta(|Q_{x,N}(\theta)|) 
\Big) e^{4 i \pi x \theta} d\theta\Big|\\
&\leq \sum_{ \overset{\delta_{0}>\delta> N_x^{-1}}{\textrm{dyadic}}}
\Big|\int_{0}^{1} \widehat{\phi}(\theta) \Big(Q_{x,N}(\theta)\sigma_\delta(\big|Q_{x,N}(\theta)\big|)\Big) \Big) e^{4 i \pi x \theta} d\theta\Big|+ 
\Big|\sum_{ \overset{\delta \not \in [\frac{1}{N_x},\delta_{0}]}{\textrm{dyadic}}}\cdots\Big|.\\
&\leq  \sum_{ \overset{\delta_{0}>\delta> N_x^{-1}}{\textrm{dyadic}}}
\Big|\int_{0}^{1} \widehat{\phi}(\theta) \Big(Q_{x,N}(\theta)\sigma_\delta(\big|Q_{x,N}(\theta)\big|)\Big) \Big) e^{4 i \pi x \theta} d\theta\Big|+ \frac{C}{\sqrt[4]{N_x}}.
\end{align}
The last inequality can be obtained by applying Lemma \ref{BW} to the functions
$\ds \varphi(\theta)=\sum_{ \overset{\delta \not \in [\frac{1}{N_x},\delta_{0}]}{\textrm{dyadic}}}Q_{x,N}(\theta) \sigma_\delta(\big|Q_{x,N}(\theta)\big|)$ and $\phi$. We further notice  that 
$$\Big\|Q_{x,N}\Big\|_2\leq \frac{1}{\sqrt{N_x}}$$
and 
$$\big\|\phi\big\|_{\infty} \leq 1.$$

\noindent{}To finish the proof, we need to point out that any bounded aperiodic mutilplicative function is statistically orthogonal to any nilsequence, by the generalized Daboussi-Delange theorem \cite[Theorem 2.5 ]{Host-N}. We further notice that for any $f,g \in L^2(X,\mu)$, and any $\varepsilon>0$, there exist $f_1,g_1 \in L^{\infty}(X,\mu)$ such that
	$$\Big\|f-f_1\Big\|_2 < \sqrt{\varepsilon}, \textrm{~~and~~} \Big\|g-g_1\Big\|_2 < \sqrt{\varepsilon}.$$
	Moreover, by Cauchy-Schwarz inequality, we have
	\begin{eqnarray*}
		&&\Big|\frac1{N}\sum_{n=1}^{N}\bmnu(n)(f-f_1)(T^{n}x)(g-g_1)(T^{-n}x) \Big|\\
		&\leq&  \frac1{N}\sum_{n=1}^{N}\big|(f-f_1)(T^{n}x)\big|\big|(g-g_1)(T^{-n}x)\big|\\
		&\leq& \Big(\frac1{N}\sum_{n=1}^{N}\big|(f-f_1)(T^{n}x)\big|^2\Big)^{\frac12}
		\Big(\frac1{N}\sum_{n=1}^{N}\big|(g-g_1)(T^{-n}x)\big|^2\Big)^{\frac12}
	\end{eqnarray*}
	Applying the ergodic theorem, it follows that for almost all $x \in X$, we have
	\[
	\limsup_{N \longrightarrow +\infty}\Big|\frac1{N}\sum_{n=1}^{N}\bmnu(n)(f-f_1)(T^{n}x)(g-g_1)(T^{-n}x) \Big|<  \varepsilon.
	\]
	Whence, we can write
	\begin{eqnarray*}
		&&\limsup_{N \longrightarrow +\infty}\Big|\frac1{N}\sum_{n=1}^{N}\bmnu(n)f(T^{n}x)g(T^{-n}x) \Big|\\
		&\leq&
		\limsup_{N \longrightarrow +\infty}\Big|\frac1{N}\sum_{n=1}^{N}
		\bmnu(n)f_1(T^{n}x)g(T^{-n}x)\Big|\\
		&+&
		\limsup_{N \longrightarrow +\infty}\Big|\frac1{N}\sum_{n=1}^{N}\bmnu(n)f(T^{n}x)g_1(T^{-n}x)\Big|\\
		&+&\limsup_{N \longrightarrow +\infty}\Big|\frac1{N}\sum_{n=1}^{N}\bmnu(n)f_1(T^{n}x)g_1(T^{-n}x)\Big|\\
		&\leq& \varepsilon+\limsup_{N \longrightarrow +\infty}\Big|\frac1{N}\sum_{n=1}^{N}
		\bmnu(n)f_1(T^{n}x)g(T^{-n}x)\Big|\\
		&+&
		\limsup_{N \longrightarrow +\infty}\Big|\frac1{N}\sum_{n=1}^{N}\bmnu(n)f(T^{n}x)g_1(T^{-n}x)\Big|.
	\end{eqnarray*}
	We thus need to estimate
	$$\limsup_{N \longrightarrow +\infty}\Big|\frac1{N}\sum_{n=1}^{N}
	\bmnu(n)f_1(T^{n}x)g(T^{-n}x)\Big|,$$
	and
	$$
	\limsup_{N \longrightarrow +\infty}\Big|\frac1{N}\sum_{n=1}^{N}\bmnu(n)f(T^{n}x)g_1(T^{-n}x)\Big|.$$
	In the same manner we can see that
	\begin{eqnarray*}
		&&\limsup_{N \longrightarrow +\infty}\Big|\frac1{N}\sum_{n=1}^{N}
		\bmnu(n)f_1(T^nx)(g-g_1)(T^{-n}x)\Big|\\
		&\leq& \limsup_{N \longrightarrow +\infty}\Big(\frac1{N}\sum_{n=1}^{N}|f_1(T^{n}x)|^2\Big)^{\frac12}
		\limsup_{N \longrightarrow +\infty}\Big(\frac1{N}\sum_{n=1}^{N}|(g-g_1)(T^{-n}x)|^2\Big)^{\frac12}\\
		&\leq& \big\|f_1\big\|_2 \big\|g-g_1\big\|_2\\
		&\leq& \Big(\big\|f\big\|_2+\sqrt{\varepsilon}\Big). \sqrt{\varepsilon}
	\end{eqnarray*}
	This gives
	\begin{eqnarray*}
		&&\limsup_{N \longrightarrow +\infty}\Big|\frac1{N}\sum_{n=1}^{N}
		\bmnu(n)f_1(T^{n}x)g(T^{-n}x)\Big|\\
		&\leq& \Big(\Big\|f\Big\|_2+\sqrt{\varepsilon}\Big). \sqrt{\varepsilon}+
		\limsup_{N \longrightarrow +\infty}\Big|\frac1{N}\sum_{n=1}^{N}
		\bmnu(n)f_1(T^{n}x)g_1(T^{-n}x)\Big|\\
		&\leq& \Big(\big\|f\big\|_2+\sqrt{\varepsilon}\Big). \sqrt{\varepsilon}+0
	\end{eqnarray*}
	Summarizing, we obtain the following estimates
	\begin{eqnarray*}
		&&\limsup_{N \longrightarrow +\infty}\Big|\frac1{N}\sum_{n=1}^{N}\bmnu(n)f(T^{n}x)g(T^{-n}x) \Big|\\
		&\leq& \varepsilon+\Big(\big\|f\big\|_2+\sqrt{\varepsilon}\Big). \sqrt{\varepsilon}+\Big(\big\|g\big\|_2+\sqrt{\varepsilon}\Big). \sqrt{\varepsilon}
	\end{eqnarray*}
	Since $\varepsilon>0$ is arbitrary, we conclude that for almost every $x \in X$,
	\[
	\frac1{N}\sum_{n=1}^{N}\bmnu(n)f(T^{n}x)g(T^{-n}x) \tend{N}{+\infty}0.
	\]
	This complete the proof of the theorem.
\end{proof}
It is noticed in \cite{IJL} that the convergence almost sure holds for the short interval can be obtained for the Liouville and M\"{o}bius functions by applying the following Zhan's estimation \cite{Ztao}: for each $A>0$, for any $\varepsilon>0$, we have
\begin{equation}\label{Zhantao}
\max_{z \in \T}\left|\displaystyle\sum_{ N \leq n \leq N+M}z^n\bml(n)\right|\leq C_{A,\varepsilon}\frac{M}{\log^{A}(M)}\text{ for some }C_{A,\varepsilon}>0,
\end{equation}
provided that $M \geq N^{\frac58+\varepsilon}$.  Here, 
\begin{xques}we ask on the convergence almost sure in the short interval for the bilinear ergodic bilinear averages 
with bounded aperiodic multiplicative weight.
\end{xques}
\begin{rem}In the forthcoming revised version of \cite{elabdal-poly}. The author will present  the strategy of the proof of Bourgain double ergodic theorem and its adaptation to prove the polynomials Bourgain bilinear ergodic
	theorem as it is stated in that paper.
	\footnote{The paper \cite{elabdal-poly} was posted on Arxiv on 5 August 2019, and
	 very recently (3 August 2020), the authors in \cite{TMT} provided a proof of a partial Bourgain bilinear ergodic theorem for polynomials  for $\sigma$-finite measure space. It is seems that therein, the authors extended some ideas introduced by the author based on the  Cald\'eron principal and Bourgain's method \cite{Bourgain-IHP} combined with the duality argument due to F. Riesz (precisely, $\ell^1$-$\ell^\infty$-duality) and the result of Peluse-Pendiville in additive combinatorics with some tools from Adelic harmonic analysis. In the revised version \cite{elabdal-poly}, the author will take into account their remarks.} 
\end{rem}
\begin{thank}
 	The author wishes to thanks Terrance Tao for his comments and suggestions and for pointing out a gap in the previous version of the proof of Theorem 5.1.
 \end{thank}

\end{document}